\documentclass[11pt,leqno,amscd,amssymb,pstricks,verbatim]{article}
\usepackage{pifont}
\usepackage{amsfonts}
\usepackage{amsmath,amsthm,amssymb}
\usepackage{amsmath, epsfig}
\usepackage{latexsym}
\usepackage{amssymb}
\usepackage{mathrsfs}
\usepackage{amscd}
\usepackage{xypic}
\usepackage[all]{xy}
\usepackage{dsfont}
\usepackage{tikz}


 
\newtheorem{thm}{Theorem}[section]

\newtheorem{lem}[thm]{Lemma}

\newtheorem{cor}[thm]{Corollary}

\newtheorem{prop}[thm]{Proposition}

\begin{document}
\title{\bf Finite groups with odd Sylow automizers}

\author{Chaida Xu and Yuanyang Zhou}

\date{\small
School of Mathematics and Statistics, Central China Normal University,
Wuhan 430079, China}
\maketitle

\bigskip\noindent{\bf Abstract} Let $p$ be an odd prime number. In this paper, we characterize the nonabelian composition factors of a finite group with odd $p$-Sylow automizers, and then prove that the McKay conjecture, the Alperin weight conjecture and the Alperin-McKay conjecture hold for such a group.

\bigskip\noindent{\bf Keywords}: Finite groups; composition factors; irreducible ordinary characters

\section{Introduction}
\quad\, Let $p$ be a prime number, $G$ a finite group and $P$ a Sylow $p$-subgroup of $G$. McKay conjecture asserts that
$|{\rm Irr}_{p'}(G)|=|{\rm Irr}_{p'}(N_G(P))|$, where ${\rm Irr}_{p'}(G)$ denotes the set of irreducible ordinary characters of $G$ with degree not divisible by $p$. Navarro conjectured that McKay conjecture should be compatible with certain Galois automorphisms (see \cite[Conjecture A]{N}). Guralnick, Malle and Navarro completely determined the composition factors of finite groups with odd self-normalizing Sylow subgroups (see \cite[Theorem 1.1]{GMN}), and then proved his conjecture for such finite groups (see \cite[Corollary B]{NTV}). Later, Guralnick, Navarro and Tiep characterized the composition factors of more general finite groups (see \cite[Theorem A]{GNT}), which have odd Sylow normalizers, and proved McKay conjecture for such finite groups. Along these developments, we investigate the composition factors of finite groups with odd Sylow automizers.

\begin{thm}\label{COMP}Let $G$ be a finite group, $p$ an odd prime number, and $P$ a Sylow $p$-subgroup of $G$. Assume that the quotient group $N_G(P)/PC_G(P)$ has odd order. If $S$ is a composition factor of $G$, then either $S$ has cyclic Sylow $p$-subgroups, or $S={\rm PSL}_2(q)$ for some power $q=p^f \equiv 3\, ({\rm mod}\, 4)$.
\end{thm}

We call $N_G(P)/PC_G(P)$ as a $p$-Sylow automizer of $G$. The composition factors of $G$ with the trivial $p$-Sylow automizer have been investigated (see \cite[Theorem 1.1 (2)]{GMN}).

Theorem \ref{COMP} looks almost the same as \cite[Theorem A]{GNT}. However there are finite simple groups with odd Sylow automizer, but with even Sylow normalizer. For example, Mathieu group $M_{24}$ and Janko group $J_4$ have odd 7-Sylow automizers and even 7-Sylow normalizers.
The proof of Theorem \ref{COMP} depends on the classification of finite simple groups.
The main difficulty of the proof lies in the case of finite simple groups of defining characteristic $p$.

We use Theorem \ref{COMP} to prove several conjectures. By \cite[Theorems B and 15.3]{IMN} and the first paragraph of the proof of \cite[Theorem 6.6]{KS} (for odd prime numbers) and by \cite{MS} (for the prime number 2), McKay conjecture holds finite groups with odd Sylow automizers.

\begin{cor}\label{Mc}Let $G$ be a finite group, $p$ a prime number, and $P$ a Sylow $p$-subgroup of $G$. Assume that $N_G(P)/PC_G(P)$ has odd order. Then we have $|{\rm Irr}_{p'}(G)|=|{\rm Irr}_{p'}(N_G(P))|$.
\end{cor}

By \cite[Theorems A and C]{Sp} and \cite[Theorem 1.1]{KS}, the blockwise Alperin weight conjecture holds for finite groups with odd Sylow automizers. This almost generalizes the main theorem of \cite{SW}.

\begin{cor}\label{Alperin} Let $G$ be a finite group, $p$ an odd prime number, and $P$ a Sylow $p$-subgroup of $G$. Assume that $N_G(P)/PC_G(P)$ has odd order. Then Alperin weight conjecture holds for all $p$-blocks of the group $G$.
\end{cor}

By \cite[Theorem C]{Sp1}, \cite[Theorem A]{CS} and \cite[Theorem 1.1]{KS}, the Alperin-McKay conjecture holds for finite groups with odd Sylow automizers.

\begin{cor}\label{AM} Let $G$ be a finite group, $p$ an odd prime number, and $P$ a Sylow $p$-subgroup of $G$. Assume that $N_G(P)/PC_G(P)$ has odd order. Then the Alperin-McKay conjecture holds for the group $G$ and the prime $p$.
\end{cor}

In the rest of the paper, we always assume that $p$ is an odd prime number. We remark that $N_G(P)/C_G(P)$ has odd order if and only if $N_G(P)/PC_G(P)$ has odd order. So we alternatively use $N_G(P)/C_G(P)$ and $N_G(P)/PC_G(P)$ during the proof of Theorem \ref{COMP}.

\section{Reduction of Theorem \ref{COMP}}

\quad\, In this section, we prepare some results for the reduction of Theorem \ref{COMP}.

\begin{lem}\label{Extension} Let $G$ be a finite group with a Sylow $p$-subgroup $P$, and $H$ a normal subgroup of $G$ such that $G=PH$. Set $Q=P\cap H$. Then we have $C_{N_H(Q)/QC_H(Q)}(P)=N_H(P)C_H(Q)/QC_H(Q)$, and the inclusion $N_H(P)\subset N_G(P)$ induces a surjective group homomorphism $$C_{N_H(Q)/QC_H(Q)}(P)\rightarrow N_G(P)/PC_G(P)$$ with kernel being a $p$-subgroup.
\end{lem}

\begin{proof} Clearly $P$ acts on $N_H(Q)/QC_H(Q)$ by conjugation, and the subgroup $C_{N_H(Q)/QC_H(Q)}(P)$ of all $P$-fixed elements in $N_H(Q)/QC_H(Q)$ is equal to $N_H(P)C_H(Q)/QC_H(Q)$. Since $G=PH$ and $N_H(P)\cap QC_H(Q)=Q C_H(P)$, the inclusion $N_H(P)\subset N_G(P)$ induces the surjective group homomorphism in the lemma. The kernel of the homomorphism is $(N_H(P)\cap PC_G(P))C_H(Q)/QC_H(Q)$, which is a $p$-group.
\end{proof}

\begin{prop}\label{Red} Assume that Theorem \ref{COMP} holds for all almost simple finite groups with odd Sylow $p$-automizer. Then Theorem \ref{COMP} holds for all finite groups with odd Sylow $p$-automizer.
\end{prop}

\begin{proof} Let $S$ be a nonabelian composition factor of $G$ and $P$ a Sylow $p$-subgroup of $G$.
Since the odd Sylow $p$-automizer condition is inherited by factor groups, we may assume that $S$ is a nonabelian simple subgroup of $G$. Since any principal series may be refined into a composition series and any two composition series are equivalent, we may further assume that there is a minimal normal subgroups $N$ of $G$ such that $S$ is normal in $N$. If $p$ does not divide the order of $N$, then $p$ does not divide the order of $S$ either. So it suffices to show Theorem \ref{COMP} when $p$ divides the order of $N$ and then that of $S$. Clearly $P$ is a Sylow $p$-subgroup of $PN$ and $N_{PN}(P)/C_{PN}(P)$ has odd order. By induction, we may assume that $G=PN$.

We write $N=S_1\times S_2\times \cdots \times S_t$ for some isomorphic simple groups $S_1, S_2, \cdots, S_t$.
Since $S$ is normal in $N$, we may assume that $S_1=S$. By \cite[3.3.10]{R}, $P$ permutably acts on the set $\{S_1, S_2, \cdots, S_t\}$. The orbit of this action containing $S$ determines a direct factor $N_1$ of $N$, $PN_1$ is a direct factor of $PN$ and $N_{PN_1}(P)/C_{PN_1}(P)$ has odd order. By induction, we may assume that $P$ acts transitively on $S_1, S_2, \cdots, S_t$.

Set $Q=P\cap N$. Clearly $N_N(P)\subset N_N(Q)$ and the inclusion induces a group homomorphism $N_N(P)/C_N(P)\rightarrow N_N(Q)/QC_N(Q)$ with kernel $QC_N(P)/C_N(P)$ and image $N_N(P)C_N(Q)/QC_N(Q)$. Note that $N_N(P)C_N(Q)/QC_N(Q)=C_{N_N(Q)/QC_N(Q)}(P/Q).$ So, $N_N(P)/C_N(P)$ has odd order if and only if so does $C_{N_N(Q)/QC_N(Q)}(P/Q)$.

Set $Q_1=Q\cap S$, $P_1=N_P(S)$ and $H=SP_1$. Let $U=\{x_1,\cdots, x_t\}$ be a representative set of right cosets of $P_1$ in $P$ and take $x_1=1$. Suppose that $C_{N_S(Q_1)/Q_1C_S(Q_1)}(P_1/Q_1)$ has an involution $xQ_1C_S(Q_1)$. It is easy to prove that the product $z=\Pi_{u\in U} x^u\in N_N(Q)$ and that $zQC_N(Q)$ centralizes $P/Q$. So $N_N(P)/C_N(P)$ has even order. This contradicts $N_G(P)/C_G(P)$ having odd order. So $N_{S_1}(P_1)/C_{S_1}(P_1)$ has odd order and so does $N_{H}(P_1)/C_{H}(P_1)$. By induction, we assume that $N=S$ and $G=SP$.

 Set $R=C_P(S)$. Then $R$ is a normal $p$-subgroup of $G$, $P/R$ is a Sylow $p$-subgroup of $G/R$ and $N_{G/R}(P/R)/C_{G/R}(P/R)$ has odd order. By induction, we assume that $P$ acts faithfully on $S$ by conjugation. Then $G$ is an almost simple group and the proof is done by the assumption of the proposition.
\end{proof}

 Let $R$ be a $p$-group. Denote by $Z(R)$ the center of $R$.

\begin{lem}\label{Real} Let $G$ be a finite group and $P$ a Sylow $p$-subgroup of $G$. Assume that $P$ is nontrivial and that $N_G(P)/PC_G(P)$ has odd order. Then all nonidentity elements in $Z(P)$ are not real.
\end{lem}

\begin{proof} Take a nonidentity element $t$ in $Z(P)$. Suppose that $t$ is real in $G$. Then $t$ and $t^{-1}$ are $G$-conjugate to each other. By \cite[Chapter 7, Theorem 1.1]{G68}, there exists $s\in N_G(P)$ such that $t^{s}=t^{-1}$. Since $t^{s^2}=t$, the order $o(s)$ of $s$ has to be even; otherwise, $t=t^{s^{o(s)}}=t^{-1}$ and $t$ is the identity of $G$. Write $o(s)=2^m r$ with $r$ an odd number. Since $t^{s^r}=t^{-1}$, $s^r$ lies in $N_G(P)$ but outside $C_G(P)$. So the image of $s^r$ in $N_G(P)/C_G(P)$ is a 2-element and $N_G(P)/C_G(P)$ has even order. That  contradicts $N_G(P)/PC_G(P)$ having odd order. So $t$ is not real. \end{proof}

\begin{cor}\label{ADE}Let $G$ be a finite group, $P$ a Sylow $p$-subgroup of $G$, and $S$ a normal nonabelian simple subgroup of $G$. Assume that $N_G(P)/PC_G(P)$ has odd order and that $p$ divides the order of $S$. Assume that $S$ is a simple group of Lie type over $F_q$, the field of $q$-elements, that $p\nmid q$, and that $S$ is not ${\rm B}_2(2)'$, ${\rm G}_2(2)'$, $^2{\rm F}_4(2)'$, or $^2{\rm G}_2(3)'$. Then $S$ is of type ${\rm A}_n$ or $^2 {\rm A}_n$ with $n\geq 2$, ${\rm D}_n$ or $^2 {\rm D}_n$ with $2\nmid n$ and with $n\geq 5$, ${\rm E}_6$ or $^2{\rm E}_6$.
\end{cor}

\begin{proof} Since $S$ is not ${\rm B}_2(2)'$, ${\rm G}_2(2)'$, $^2{\rm F}_4(2)'$, or $^2{\rm G}_2(3)'$,
there is a simply-connected simple algebraic group $\cal G$ over the algebraic closure $\bar F_q$ of $F_q$ with a Frobenius map $F: \cal G\rightarrow \cal G$, such that $S$ is isomorphic to ${\cal G}^F/Z({\cal G}^F)$. Assume that $S$ is not of any type in the corollary. Then $\cal G$ is not of type ${\rm A}_n$ with $n\geq 2$, ${\rm D}_{2n+1}$, or ${\rm E}_6$. By \cite[Proposition 3.1 (i) and (ii)]{TZ}, all semisimple elements in ${\cal G}^F$ are real. Since $p\nmid q$, all $p$-elements in ${\cal G}^F$ are semisimple and real. Thus all $p$-elements in $S$ are real. But by the assumption and Lemma \ref{Real}, all nonidentity elements in $Z(P)\cap S$ are not real in $G$. That causes a contradiction. \end{proof}

\section{Proof of Theorem \ref{COMP}}

\quad\, In order to prove Theorem \ref{COMP}, by Proposition \ref{Red}, it suffices to show the following theorem.

\begin{thm}\label{Almost}Let $G$ be an almost simple finite group, $P$ a Sylow $p$-subgroup of $G$, and $S$ the socle of $G$. Assume that $G=SP$ and that $N_G(P)/PC_G(P)$ has odd order. Then either $S$ has cyclic Sylow $p$-subgroups, or $S={\rm PSL}_2(q)$ for some power $q=p^f \equiv 3\, ({\rm mod}\, 4)$.
\end{thm}

\begin{lem}\label{Alternating} Theorem \ref{Almost} holds when $S$ is an alternating group $A_n$.
\end{lem}

\begin{proof} When $S$ is $A_n$, the group $G$ is equal to $S$. Given an element $t$ in $S$ with order $p$, we write $t$ as a product $(a_{11}a_{12}\cdots a_{1p})(a_{21}a_{22}\cdots a_{2p})\cdots(a_{k1}a_{k2}\cdots a_{kp})$
of $k$ disjoint cycles. When $k\geq 2$, $t$ is centralized by $(a_{11}a_{21})(a_{12}a_{22})\cdots (a_{1p}a_{2p})$; when
$n-kp\geq 2$, $t$ is centralized by an involution $(ij)$ with $i, j\in \{1,2, \cdots, n\}-\{a_{11}, a_{12}, \cdots, a_{1p}, a_{21}, a_{22}, \cdots, a_{2p}, a_{k1}, a_{k2}, \cdots, a_{kp}\}$. So, when $k\geq 2$ or $n-kp\geq 2$, the conjugacy class of $t$ in $A_n$ is equal to the conjugacy class of $t$ in the Symmetric group $S_n$; in particular, $t$ and $t^{-1}$ are conjugate in $A_n$ and $t$ is real. But by Lemma \ref{Real} and the assumption in Theorem \ref{Almost}, all nonidentity elements in $Z(P)$ are not real. So $k\leq 1$ and $n-kp\leq 1$. This implies that $n=p$ or $n=p+1$ and thus that Sylow $p$-subgroups of $S$ have to be cyclic.
\end{proof}

\begin{lem}\label{Sporadic} Theorem \ref{Almost} holds when $S$ is a sporadic simple group.
\end{lem}

\begin{proof} Since $p$ is odd, we have $G=S$ and $P=Q$. Clearly the $N_G(P)$-conjugation induces an action of $N_G(P)/P$ on the set $I$ of all nontrivial irreducible ordinary character of $P/P'$, where $P'$ is the derived subgroup of $P$. The stabilizer of any character $\lambda\in I$ contains $PC_G(P)$, so the length of the orbit of the action containing $\lambda$ divides the order of $N_G(P)/PC_G(P)$. When $P$ is noncyclic, by checking \cite{W}, the length of the orbit containing $\lambda$ for any $\lambda\in I$ is always even. By the assumption, $P$ has to be cyclic.
\end{proof}

We use the notation ${\rm SL}^\epsilon$, ${\rm GL}^\epsilon$, ${\rm PGL}^\epsilon$ and ${\rm PSL}^\epsilon$ to denote ${\rm SL}$, ${\rm GL}$, ${\rm PGL}$ and ${\rm PSL}$ respectively when $\epsilon=+$, and ${\rm SU}$, ${\rm GU}$, ${\rm PGU}$ and ${\rm PSU}$ respectively when $\epsilon=-$. Suppose that $L={\rm SL}_n^\epsilon(q)$ and $H={\rm GL}_n^\epsilon(q)$ for a power $q$ of some prime $r$. The group $H$ is isomorphic to ${\rm GL}^\epsilon(V)$, where $V$ is $F_q^n$ when $\epsilon=+$ and $F_{q^2}^n$ when $\epsilon=-$. Here $F_q$ denotes the field of $q$-elements. We fix a basis $e_1,\cdots, e_n$ of $V$, and suppose that it is orthonormal if $\epsilon=-$. We suppose that if
$(x_{ij})\in H$ corresponds to $f\in {\rm GL}^\epsilon(V)$ under the isomorphism, then $(x_{ij})$ is the matrix of $f$ under the basis $e_1,\cdots, e_n$.
Denote by $\sigma$ the field automorphism $H\rightarrow H, (x_{ij})\mapsto (x_{ij}^r)$. The isomorphism $\sigma$ induces field automorphisms on ${\rm SL}_n^\epsilon(q)$, ${\rm GL}_n^\epsilon(q)$, ${\rm PGL}_n^\epsilon(q)$ and ${\rm PSL}_n^\epsilon(q)$, and we denote these automorphisms by $\sigma$ for convenience. Let $\sigma_0$ be the $p$-part of $\sigma$.

Suppose that $p$ and $r$ are different. Let $R$ be a Sylow $p$-subgroup of $H$. By \cite[Chapter 3, \S10]{GL} and \cite[Chapter 4, \S4.10]{GLS}, we have $R=R_T\rtimes R_W$; there is a decomposition (which is orthogonal if $\epsilon=-$) $$V=V_0\oplus V_1\oplus\cdots \oplus V_m$$ of $V$ (compatible with the given basis) as $F_q R_T$-modules (as $F_{q^2} R_T$-modules if $\epsilon=-$), such that $V_i\cong V_1$ for $1\leq i\leq m$, ${\rm dim} V_i=e={\rm ord}_p(\epsilon q)$ and $0\leq {\rm dim} V_0<e$; $R_T=R_1\times \cdots \times R_m$ with $R_i$ a cyclic subgroup of a cyclic maximal torus $T_i$ of ${\rm GL}^\epsilon(V_i)$ and $|T_i|=q^e-\epsilon^e$; there is a subgroup $\Sigma$ of $N_H(R_T)$, such that $\Sigma$ is centralized by $\sigma$ and isomorphic to $S_m$ and it acts on the sets $\{V_1,\cdots, V_m\}$, $\{T_1,\cdots, T_m\}$ and $\{R_1,\cdots, R_m\}$ by permuting subscripts; $R_W$ is a Sylow $p$-subgroup of $\Sigma$.

Clearly $\sigma_0$ stabilizes ${\rm GL}^\epsilon(V_1)$ and $R_W$. We choose $R_1$ to be a $\sigma_0$-stable $p$-Sylow subgroup of ${\rm GL}^\epsilon(V_1)$ and then construct $R_i$ with the isomorphism $V_i\cong V_1$ for any $1\leq i\leq m$. Then $R_T$ is $\sigma_0$-stable and so are $R$ and $LR$. Set $H^\ast=(LR)\rtimes \langle \sigma_0 \rangle$, $P^\ast=R\rtimes \langle \sigma_0 \rangle$ and $Q^\ast=L\cap P^\ast$.

\begin{lem}\label{A} Theorem \ref{Almost} holds when $S={\rm PSL}_n^\epsilon(q)$ with $n\geq 3$, $\epsilon=\pm$ and $p\nmid q$.
\end{lem}

\begin{proof} Since $G=SP$ is an almost simple group and $p>2$, by the structure of the automorphism group of $S$, we assume that $G$ is a subgroup of ${\rm PGL}_n^\epsilon(q)\rtimes \langle \sigma\rangle$. Set $M=LR\cap Z({\rm GL}_n^\epsilon(q))$. The inclusion $H^\ast\subset {\rm GL}_n^\epsilon(q)\rtimes \langle \sigma\rangle$ induces an injective group homomorphism $H^\ast/M\rightarrow {\rm PGL}_n^\epsilon(q)\rtimes \langle \sigma\rangle$, so that we may identify $H^\ast/M$ with a subgroup of ${\rm PGL}_n^\epsilon(q)\rtimes \langle \sigma\rangle$.
Clearly $S$ is contained in $H^\ast/M$.
Since $P^\ast$ is a Sylow $p$-subgroup of ${\rm GL}_n^\epsilon(q)\rtimes \langle \sigma\rangle$, by a suitable conjugation,
we may assume that $G$ and $P$ are contained in $H^\ast/M$ and $P^\ast M/M$ respectively.

{\bf  a.} Suppose $p\mid(q-\epsilon)$. Then $e=1$ and $m=n$. Let $V_i$ be the subspace generated by $e_i$. Without loss of generality, we identify $\Sigma$ with $S_n$ and suppose that $S_n$ acts on the set $\{e_1, \cdots, e_n\}$ by permuting subscripts.

Suppose that there is an involution $s\in N_{S_n}(R_W)$. Take $t$ to be $s$ if the determinant of $s$ is 1 and $s{\rm diag}(1, \cdots, 1, -1)$ if the determinant of $s$ is $-1$. Note that $t\in L$, that $t$ stabilizes $R_T$ and $R$, that $t$ and $\sigma_0$ commute, and that $t$ normalizes $Q^\ast$. For any $x\in R_1$, we have $txt^{-1}x^{-1}\in L\cap R_T\leq Q^\ast$.
Since $P^\ast=Q^\ast R_1 \langle\sigma_0 \rangle$, $t$ centralizes $P^\ast/Q^\ast$. Since the inclusion $P^\ast\subset P^\ast M$ induces a surjective homomorphism $P^\ast/Q^\ast\rightarrow (P^\ast M/M)/(Q^\ast M/M)$, the image $\bar t$ of $t$ in $S$ centralizes the quotient group $(P^\ast M/M)/(Q^\ast M/M)$. Since the inclusion $P\subset P^\ast M/M$ induces an injective group homomorphism $P/Q\rightarrow (P^\ast M/M)/(Q^\ast M/M)$, $\bar t$ centralizes $P/Q$. So we have $\bar tQ\in C_{N_S(Q)/Q}(P/Q)=N_S(P)/Q$.

Suppose $\bar t\in C_S(P)Q$. Then $\bar t\in C_S(P)\subset C_S(Q)$. Since $Q=Q^\ast M/M$, we have $[t, Q^\ast]\subset Z({\rm GL}^\epsilon(q))$. We write $t$ into a product $(i_1i_2)(j_1j_2)\cdots$ of transpositions with disjoint letters. Without loss of generality, we suppose $(i_1i_2)=(12)$. Take some $x={\rm diag}(a^{-1}, a, 1, \cdots, 1)\in Q^\ast$ for some $a\neq 1$.
Since $txt^{-1}x^{-1}=(12)x(12)^{-1}x^{-1}={\rm diag}(a^2, a^{-2}, 1, \cdots, 1)\in Z({\rm GL}^\epsilon(q))$, we have $a^2=1$. Since $p$ is odd, $a=1$. That causes a contradiction. So $\bar t$ is not in $C_S(P)Q$, the image of $\bar t$ in $N_S(P)/C_S(P)$ is an involution,  $N_S(P)/C_S(P)$ has even order and so does $N_G(P)/C_G(P)$.

We prove that such $s$ exists when $p\mid(q-\epsilon)$. By the proof in the last two paragraphs, $N_G(P)/C_G(P)$ has even order. That contradicts the assumption in Theorem \ref{Almost}.

Suppose $n\geq p+2$. By \cite[Lemma 3.3(i)]{GNT}, there is an involution $s\in N_{A_n}(R_W)$.
 Suppose $n=p ~ \mbox{\rm or}~ p+1$. We choose $R_W$ to be the subgroup generated by the cycle $(12\cdots p)$ and can take $s=(1, p)(2, p-1)\cdots (\frac{p-1}{2}, \frac{p+3}{2})$.
Suppose $4\leq n<p$. Then $R_W$ is trivial and we can take $s=(12)(34)$.
Suppose $n=3$. We can take $s=(12)$.

{\bf b.} Suppose $p\nmid(q-\epsilon)$. Then $R\lneqq L$ and $e>1$. When $m=1$, $R$ is cyclic and so is Sylow $p$-subgroup of $S$.
So we may assume that $m\geq 2$. We identify $\Sigma$ with $S_n$.

By the proof of \cite[Proposition 3.5]{GNT}, there is an involution $t\in N_{A_m}(R_W)$, except when $p>m=2, 3$ and when $m=p~\mbox{\rm or}~ p+1$. In this case, by a proof similar to the proof in the second and third paragraphs in the case {\bf a}, we prove that $N_G(P)/C_G(P)$ has even order.

When $p>m=2, 3$ or $m=p~\mbox{\rm or}~ p+1$, we take $t$ to be the element $t$ in the corresponding case in the proof of \cite[Proposition 3.5]{GNT}. Again, by a proof similar to the proof in the second and third paragraphs in the case {\bf a}, we prove that $N_G(P)/C_G(P)$ has even order.
\end{proof}

Let $V=F_q^{2n}$ be a quadratic space of type $\epsilon$, where $q$ is a power of some prime $r$. Set $H={\rm GO}(V)$ and $L=\Omega(V)$.
Choose an orthogonal basis $e_1,\cdots, e_n$ of $V$. For any $f\in H$, denote by $(x_{ij})_f$ the matrix under the basis $e_1,\cdots, e_n$. The map $H\rightarrow {\rm GO}_{2n}^\epsilon(q), f\mapsto (x_{ij})_f$ is a group isomorphism and it sends $L$ onto $\Omega_{2n}^\epsilon(q)$. Denote by $\sigma$ the field automorphism $H\rightarrow H, (x_{ij})\mapsto (x_{ij}^r)$. The isomorphism $\sigma$ induces field automorphisms on $L$ and ${\rm P\Omega}_{2n}^\epsilon(q)$, still denoted by $\sigma$. Let $\sigma_0$ be the $p$-part of $\sigma$.

Suppose that $p$ and $r$ are different. Define $\epsilon_p=+$ if $e={\rm ord}_p(q)$ is odd, and $\epsilon_p=-$ if $e$ is even. Set $d={\rm lcm}(2, e)$.
Let $R$ be a Sylow $p$-subgroup of $H$. By \cite[Chapter 3, \S10]{GL} and \cite[Chapter 4, \S4.10]{GLS}, we have $R=R_T\rtimes R_W$; there is an orthogonal indecomposable decomposition $$V=V_0\bot V_1\bot\cdots \bot V_m$$ of $V$ as ${F_q}R_T$-modules, such that $V_i\cong V_1$ for any $1\leq i\leq m$ is a quadratic space of dimension $d$ and type $\epsilon_p$ and either ${\rm dim} V_0<d$ or ${\rm dim} V_0=d$ and $V_0$ has type $-\epsilon_p$; the basis $e_1,\cdots, e_n$ of $V$ can be chosen, so that it is compatible with the decomposition and the isomorphisms $V_i\cong V_1$ for all $i$;
$R_T=R_1\times \cdots \times R_m$ with $R_i$ a cyclic subgroup of a cyclic maximal torus $T_i$ of ${\rm GO}(V_i)$, acting orthogonally indecomposably on $V_i$ and trivially on $V_j$ for all $j\neq i$, and $|T_i|=q^{d/2}-\epsilon_p$; there is a subgroup $\Sigma$ of $N_H(R_T)$, such that $\Sigma$ is centralized by $\sigma$ and isomorphic to $S_m$ and it acts on the sets $\{V_1,\cdots, V_m\}$, $\{T_1,\cdots, T_m\}$ and $\{R_1,\cdots, R_m\}$ by permuting subscripts; $R_W$ is a Sylow $p$-subgroup of $\Sigma$.

Clearly $\sigma_0$ stabilizes ${\rm GL}(V_1)$ and $R_W$. We choose $R_1$ to be a $\sigma_0$-stable $p$-Sylow subgroup of ${\rm GO}(V_1)$ and construct $R_i$ with the isomorphism $V_i\cong V_1$ for any $1\leq i\leq m$. Then $R_T$ is $\sigma_0$-stable and so is $R$.

\begin{lem}\label{D} Theorem \ref{Almost} holds when $S={\rm P\Omega}_{2n}^\epsilon(q)$ with $2\nmid n\geq 5$, $\epsilon=\pm$ and $p\nmid q$.
\end{lem}

\begin{proof} Since $p$ is odd, we may identify $Q$ with $R$. Since $G=SP$ and $n\geq 5$, by the structure of the automorphism group of $S$, we assume that $G$ is a subgroup of $S\rtimes \langle \sigma_0\rangle$ and that $P$ is a subgroup of $P^\ast=R\rtimes \langle \sigma_0 \rangle$.

Suppose that $V_0\neq 0$, or $p\nmid m$ and $m>1$. By the first and second paragraph of the part (b) of the proof of \cite[Proposition 3.6]{GNT}, some nonidentity element of $Z(P)$ is real in $G$. But by Lemma \ref{Real} and the assumption in Theorem \ref{Almost}, all nonidentity element of $Z(P)$ are not real in $G$. So a contradiction arises. Suppose that $m=1$. Then $Q$ is cyclic.

Suppose that $V_0=0$, $m\geq 2$ and $p|m$. By the parts (c), (d) and (e) of the proof of \cite[Proposition 3.6]{GNT}, there is a $\sigma_0$-fixed involution $t\in L-Z(L)$ such that $(-1_V)^jt$ for some suitable $j\in \{0, 1\}$ belongs to $N_{S_m}(R_W)$, where $1_V$ is the identity map on $V$.

Denote by $\bar t$ the image of $t$ in $S$. Clearly $t$ normalizes $R_T$ and $R$. Since $P=R(P\cap \langle \sigma_0 \rangle)$,
$\bar t\in N_S(P)$. Suppose that $\bar t\in C_S(P)$. Then $\bar t\in C_S(P)\subset C_S(Q)$ and $[t, R]\subset \{\pm 1_V\}$. Next we use
the inclusion $[t, R]\subset \{\pm 1_V\}$ to produce a contradiction. We may assume that $t$ is an involution in $S_m$. We write $t$ into a product $(i_1i_2)(j_1j_2)\cdots$ of transpositions with disjoint letters. Without loss of generality, we suppose $(i_1i_2)=(12)$. Take a suitable $A\in {\rm GO}_d^{\epsilon_p}(q)$ of order $p$, so that
$x={\rm diag}(A^{-1}, A, E_d, \cdots, E_d)\in R$, where $E_d$ is the identity $(d\times d)$-matrix. Since $txt^{-1}x^{-1}=(12)x(12)^{-1}x^{-1}={\rm diag}(A^2, A^{-2}, E_d, \cdots, E_d)\in \{\pm E_{2n}\}$, we have $A^4=E_d$. Thus $A=E_d$ since $p$ is odd. That causes a contradiction.
Therefore $\bar t$ is outside $C_S(P)$, $N_S(P)/C_S(P)$ has even order and so does $N_G(P)/C_G(P)$. That contradicts the assumption in Theorem \ref{Almost}.
\end{proof}

\begin{lem}\label{D} Theorem \ref{Almost} holds when $S={\rm E}_6^\epsilon(q)$ with $\epsilon=\pm$ and $p\nmid q$.
\end{lem}

\begin{proof} Suppose $p\nmid (q^5-\epsilon)(q^9-\epsilon)$; by the first paragraph of the proof of \cite[Proposition 3.8]{GNT}, $Z(P)\cap Q$ contains a nonidentity real element in $S$. But, since $N_G(P)/PC_G(P)$ has odd order, by Lemma \ref{Real}, all nonidentity element in $Z(P)$ are not real. That causes a contradiction.

Suppose $p\mid (q^5-\epsilon)(q^9-\epsilon)$ but $p\nmid (q^3-\epsilon)$. Then $Q$ is cyclic.

Suppose $p\mid (q^3-\epsilon)$ but $p\nmid (q-\epsilon)$. Then $p\neq 3$; otherwise, $q^2+\epsilon q+1=(q-\epsilon)^2+ 3\epsilon q$,  $p\nmid (q^2+\epsilon q+1)$ and $p\nmid (q^3-\epsilon)$; that contradicts the assumption. Since $G=SP$ is an almost simple group, by the structure of the group ${\rm Aut}(S)$, we may assume that $G=S\rtimes \langle \varphi\rangle$ and $P=Q\rtimes \langle \varphi\rangle$ for a suitable field automorphism $\varphi$ of $S$.

Set $\hat S ={\rm E}_6^\epsilon(q)_{\rm sc}$. Then $\hat S$ is a universal cover of $S$. By \cite[Remark 2.4.11 and Theorem 2.5.14 (d)]{GLS}, the field automorphism $\varphi$ can be uniquely lifted to a field automorphism $\hat\varphi$ of $\hat S$. Let $\hat Q$ be the Sylow $p$-subgroup of the converse image of $Q$ in $\hat S$. Clearly $\varphi$ normalizes $Q$, $\hat\varphi$ normalizes $\hat Q$, the action of $\varphi$ on $S$ induces an action of $\varphi$ on $N_S(Q)/QC_S(Q)$, and the action of $\hat\varphi$ on $\hat S$ induces an action of $\hat\varphi$ on $N_{\hat S}(\hat Q)/\hat QC_{\hat S}(\hat Q)$.
The natural surjective homomorphism $\hat S\rightarrow S$ induces a group isomorphism $$N_{\hat S}(\hat Q)/\hat QC_{\hat S}(\hat Q)\cong N_S(Q)/QC_S(Q)$$ compatible with field automorphisms $\varphi$ and $\hat\varphi$. In particular, an element in $N_{\hat S}(\hat Q)/\hat QC_{\hat S}(\hat Q)$ is fixed by $\hat\varphi$ if and only if the corresponding element in $N_S(Q)/QC_S(Q)$ is fixed by $\varphi$.

Suppose that $\epsilon=+$. We claim that $p\neq 5$. Otherwise, since $p\nmid (q-1)$, we have $q=5k+t$, where $t=\mbox{\rm 0, or 2, or 3, or 4}$; since $q^3-1\equiv t^3-1\, ({\rm mod} \, 5)$, $p\nmid (q^3-1)$; that contradicts the assumption $p\mid (q^3-\epsilon)$. Since $p$ does not divide the order of the Weyl group of $\hat S$, by the proof of \cite[Proposition 7]{S}, $\hat\varphi$ fixes an involution in $N_{\hat S}(\hat Q)/C_{\hat S}(\hat Q)$. So $\hat\varphi$ fixes an involution in $N_{\hat S}(\hat Q)/\hat QC_{\hat S}(\hat Q)$ and $\varphi$ fixes an involution in $N_S(Q)/QC_S(Q)$. In particular, $P$ fixes an involution in $N_S(Q)/QC_S(Q)$. By Lemma \ref{Extension},
we have a surjective homomorphism $C_{N_S(Q)/QC_S(Q)}(P)\rightarrow N_G(P)/PC_G(P)$ with kernel being a $p$-subgroup.
The image in $N_G(P)/PC_G(P)$ of the involution fixed by $P$ is an involution. That contradicts the assumption in Theorem \ref{Almost}.

Suppose that $\epsilon=-$. Similarly we use the proof of \cite[Proposition 7]{S} to prove that $N_G(P)/PC_G(P)$ contains an involution. That contradicts the assumption in Theorem \ref{Almost} again.

Suppose that $3\neq p|(q-\epsilon)$. It is proved in the proof of \cite[Proposition 3.8]{GNT} that $Z(P)\cap Q$ contains a nonidentity real element $z$ in $S$. By Lemma \ref{Real} and Theorem \ref{Almost}, that causes a contradiction.

Suppose that $3=p|(q-\epsilon)$. Let $\cal G$ be a simply connected simple algebraic group of type ${\rm E}_6$ over $\bar F_q$ with a Frobenius map $F: \cal G\rightarrow \cal G$, such that $\hat S={\cal G}^F$. By \cite[Theorem 25.11]{MT}, there is a unique $\hat S$-conjugacy class of maximal tori $\cal T$ in $\cal G$ such that $T={\cal T}^F$ has order $(q-\epsilon)^6$ and that $N_{
\hat S}({\cal T})/T$ is the Weyl group of $\rm E_6$. By the Frattini argument, we may assume that $P$ normalizes $A=N_{\hat S}(\cal T)$ and $T$. By the last paragraph of the proof of \cite[Proposition 3.8]{GNT}, $A$ has a subgroup $B$ containing $T$ and fixed by $P$, such that $N_B(\hat Q)=\hat Q\rtimes C_2=\hat Q\rtimes \langle t\rangle$. The image of $t$ in $N_{\hat S}(\hat Q)/\hat QC_{\hat S}(\hat Q)$ is an involution fixed by $P$ and the corresponding element in $N_S(Q)/QC_S(Q)$ is an involution fixed by $P$. As above, we use Lemma \ref{Extension} to prove that $N_G(P)/PC_G(P)$ contains an involution. That causes a contradiction.
\end{proof}

\begin{lem}\label{Defining} Theorem \ref{Almost} holds when $S$ is a simple group of Lie type in characteristic $p$.
\end{lem}

\begin{proof} Suppose that $S$ is $^2 G_2(3)'$. In this case, $p=3$ and $S$ is isomorphic to ${\rm PSL}_2(8)$. Since Sylow $3$-subgroups of $S$ are cyclic, Theorem \ref{Almost} holds. So in the rest of the proof, we may suppose that $S$ is not $^2 G_2(3)'$.

Suppose that $S\neq {\rm D}_4(q)$ or $p\neq 3$, where $q$ is a power of $p$. By the structure of automorphism groups of finite simple groups, we may suppose that $G=S\rtimes \langle \delta\rangle$ and $P=Q\rtimes \langle \delta\rangle$ for a suitable field automorphism $\delta$ of order a power of $p$.
Since $p$ is odd, it is known that there is a simply connected simple algebraic group $\cal G$ over $\bar F_q$ with a Frobenius map $F: \cal G\rightarrow \cal G$ such that $S={\cal G}^F/Z({\cal G}^F)$.
By \cite[Remark 2.4.11 and Theorem 2.5.14 (d)]{GLS}, the field automorphism $\delta$ can be uniquely lifted to a field automorphism $\sigma$ of ${\cal G}^F$. Denote by $R$ the Sylow $p$-subgroup of the converse image of $Q$ in ${\cal G}^F$.
Since $\delta$ stabilizes $Q$, $\sigma$ stabilizes $R$. The natural surjective homomorphism ${\cal G}^F\rightarrow S$ induces a group isomorphism $$N_{{\cal G}^F}(R)/RC_{{\cal G}^F}(R)\cong N_S(Q)/QC_S(Q). $$

An element in $N_{{\cal G}^F}(R)/RC_{{\cal G}^F}(R)$ is fixed by $\sigma$ if and only if its image in $N_S(Q)/QC_S(Q)$ is fixed by $\delta$. In particular, an involution $d$ in $N_{{\cal G}^F}(R)/RC_{{\cal G}^F}(R)$ is fixed by $\sigma$ if and only if its image $\bar d$ in $N_S(Q)/QC_S(Q)$ is fixed by $\delta$ if and only if $\bar d\in C_{N_S(Q)/QC_S(Q)}(P)$. By Lemma \ref{Extension},
we have a surjective homomorphism $C_{N_S(Q)/QC_S(Q)}(P)\rightarrow N_G(P)/PC_G(P)$ with kernel being a $p$-subgroup.
The image in $N_G(P)/PC_G(P)$ of $\bar d$ is an involution. So in order to prove the lemma, it suffices to show that there is a $\sigma$-fixed involution in $N_{{\cal G}^F}(R)/RC_{{\cal G}^F}(R)$.

Let $T$ be a $F$-stable maximal torus of $\cal G$ and $\Phi$ the root system of $\cal G$ with respect to $T$ (a subset of the character group of $T$). Denote by ${\rm G_a}$ the additive group. For each $\alpha$, there is a morphism $x_\alpha: {\rm G_a}\rightarrow \cal G$ of algebraic groups, which induces an isomorphism onto $x_\alpha({\rm G_a})$ such that
$tx_\alpha(c)t^{-1}=x_\alpha(\alpha(t)c)$ for any $t\in T$ and any $c\in \bar F_q$. The image of $x_\alpha$ is called the root subgroup of $\cal G$ with respect to $T$ associated with $\alpha$, denoted by $U_\alpha$. For any $\alpha\in \Phi$ and any nonzero $t\in \bar F_q$,
set \begin{center} $n_\alpha(t)=x_\alpha(t)x_{-\alpha}(-t^{-1})x_\alpha(t)$ and $h_\alpha(t)=n_\alpha(t)n_\alpha(1)^{-1}$.  \end{center}
Then $T=\langle h_\alpha(t)|t\in \bar F_q-\{0\},~ \alpha\in \Phi \rangle$. Let $\Delta$ be a fundamental system of $\Phi$ and $\Phi^+$ the corresponding positive subsystem of $\Phi$. Set $U=\langle U_\alpha|\alpha\in  \Phi^+\rangle$ and $B=TU$. Then $B$ is a $F$-stable Borel subgroup of $\cal G$ containing $T$ with the unipotent radical $U$. It is known that $U^F$ is a Sylow $p$-subgroup of ${\cal G}^F$. Since $R$ is also a Sylow $p$-subgroup of ${\cal G}^F$, we may assume that $R=U^F$.

Suppose that the permutation on the coxeter graph associated with $\Phi$ induced by $F$ is trivial. Further assume that the rank of $\Phi$ is bigger than 1. Then we can choose two fundamental roots $r\neq s\in \Phi$ such that $h_r(-1)x_s(u)h_r(-1)^{-1}=x_s(-u)$ for any $u\in F_q$. So $h_r(-1)$ lies in $N_{{\cal G}^F}(R)$ but outside $C_{{\cal G}^F}(R)$. In particular, the image of $h_r(-1)$ in $N_{{\cal G}^F}(R)/RC_{{\cal G}^F}(R)$ is an involution. Since $\sigma(h_r(-1))=h_r(-1)$, the involution is fixed by $\sigma$ and so the group $N_G(P)/PC_G(P)$ has even order. This contradicts the assumption in Theorem \ref{Almost}. So the rank of $\Phi$ has to be 1 and then $S$ is isomorphic to ${\rm PSL}_2(q)$. In this case, by \cite[Theorem A]{GNT}, $q=p^f \equiv 3\, ({\rm mod}\, 4)$.

Suppose that the permutation $\rho$ on the coxeter graph associated with $\Phi$ induced by $F$ is nontrivial. In this case, since $p$ is odd, $\Phi$ is type of $\rm A$, $\rm D$, $\rm E$ or ${\rm G}_2$.

Suppose that $\Phi$ is type of $\rm A_\ell$. Further suppose that $\ell\geq 4$. We choose roots $r$ and $s$ in the coxeter graph associated with $\Phi$ as the following
$$
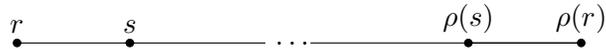
\begin{tikzpicture}[scale=1]
        \pgfmathsetmacro{\h}{1.5}
        \coordinate (a) at (5*\h,0);
        \draw[fill=black](0,0)circle[radius=.05] node[above]{$r$}--++(\h,0)circle[radius=.05]node[above]{$s$}-- ++(1.2*\h,0)node[right]{$\ldots$};
        \draw[fill=black](a)circle[radius=.05]node[above]{$\rho(r)$} --++(-\h,0)circle[radius=.05]node[above]{$\rho(s)$} ;
        \draw(a)--++(-2.4*\h,0);
    \end{tikzpicture}
$$
Set $h=h_s(-1)h_{\rho(s)}(-1)$ and $g=x_r(u)x_{\rho(r)}(u^q)$ for some $u\in F_{q^2}-\{0\}$. Clearly $h$ is a $\sigma$-fixed involution in $T^F$, $g$ is an element of $R$ and $g^h=x_r(-u)x_{\rho(r)}(-u^q)\neq g$. So the image of $h$ in $N_{{\cal G}^F}(R)/RC_{{\cal G}^F}(R)$ is a $\sigma$-fixed involution and $N_G(P)/PC_G(P)$ has even order.

Suppose that $\ell=3$ and that $\{r, s, \rho(r)\}$ is a fundamental system of $\Phi$. Set $h=h_s(-1)$ and $g=x_r(u)x_{\rho(r)}(u^q)$ for some $u\in F_{q^2}-\{0\}$. Similarly, the image of $h$ in $N_{{\cal G}^F}(R)/RC_{{\cal G}^F}(R)$ is a $\sigma$-fixed involution and $N_G(P)/PC_G(P)$ has even order.

Suppose that $\ell=2$ and that $\{r, s\}$ is a fundamental system of $\Phi$. Any element $g$ of $R$ can be uniquely written as $x_r(t_1)x_s(t_1^q)x_{r+s}(t_2)$ for some $t_1, t_2\in F_{q^2}$ such that $t_2+t_2^q=\pm t_1^{q+1}$; moreover, the order of $R$ is $q^3$.
Set $h=h_r(-1)h_s(-1)$. Suppose that $h\in C_{{\cal G}^F}(R)$. Then we have $$g^h=x_r(-t_1)x_s(-t_1^q)x_{r+s}(t_2)=x_r(t_1)x_s(t_1^q)x_{r+s}(t_2)=g. $$ The equalities force $t_1=0$. But $t_1$ may take nonzero element in $F_{q^2}$. That causes a contradiction.

Suppose that $\Phi$ is type of $\rm E_6$. Then, as in the case ${\rm A}_\ell$ for $\ell\geq 4$, we prove that $N_G(P)/PC_G(P)$ has even order. That causes a contradiction.

Suppose that $\Phi$ is type of $\rm D_\ell$ for $\ell \geq 4$. Suppose that $\ell\geq 5$ or that $S=~^2D_4(q)$. As in the case ${\rm A}_\ell$ for $\ell= 3$, we prove that $N_G(P)/PC_G(P)$ has even order. Suppose that $S=~^3D_4(q)$. We take the coxeter graph associated with $\Phi$ as the following
$$
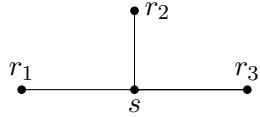
\begin{tikzpicture}[scale=1]
        \pgfmathsetmacro{\h}{1.5}
        \coordinate (a) at (5*\h,0);
        \draw[fill=black](a)circle[radius=.05]node[above]{$r_3$} --++(-\h,0)circle[radius=.05]node[below]{$s$} --+(0,.7*\h) circle[radius=.05]node[right]{$r_2$};
        \draw[fill=black](a)--++(-2*\h,0)circle[radius=.05]node[above]{$r_1$};
    \end{tikzpicture}
$$
Take $h=h_{r_1}(-1)h_{r_2}(-1)h_{r_3}(-1)\in T^F$ and $g=x_s(u)\in R$ for some nonzero $u\in F_q$. Since $g^h=x_s(-u)$ and $\sigma(h)=h$,
the image of $h$ in $N_{{\cal G}^F}(R)/RC_{{\cal G}^F}(R)$ is a $\sigma$-fixed involution and $N_G(P)/PC_G(P)$ has even order.
In a word, $N_G(P)/PC_G(P)$ always has even order when $\Phi$ is type of $\rm D_\ell$ for $\ell \geq 4$.
That causes a contradiction.

Suppose that $\Phi$ is type of $\rm G_2$. In this case, $p=3$ and $q=3^{2m+1}$. We suppose that $\{a, b\}$ is a fundamental system of $\Phi$ with $b$ the short root. Any element of $R$ can be uniquely written as a product $x_a(u)x_b(u^\theta)x_{a+3b}(v)x_{a+b}(v^\theta-u^{1+\theta})x_{2a+3b}(w)x_{a+2b}(w^\theta-u^{1+2\theta})$, where $u, v, w\in F_{}$ and $\theta$ is a field automorphism on $F_{q}$ such that $3\theta^2=1$. Take $h=h_a(-1)h_b(-1)\in T^F$. If $h$ lies in $C_{{\cal G}^F}(R)$, we have $$\begin{array}{lll}
 g^h&=& x_a(-u)x_b(-u^\theta)x_{a+3b}(v)x_{a+b}(v^\theta-u^{1+\theta})x_{2a+3b}(-w)x_{a+2b}(-w^\theta+u^{1+2\theta})\\
   & = & x_a(u)x_b(u^\theta)x_{a+3b}(v)x_{a+b}(v^\theta-u^{1+\theta})x_{2a+3b}(w)x_{a+2b}(w^\theta-u^{1+2\theta})  .
\end{array}$$
Then by the uniqueness, we have $u=w=0$. That is impossible. So $h\in N_{{\cal G}^F}(R)-C_{{\cal G}^F}(R)$. Since $\sigma(h)=h$, the image of $h$ in $N_{{\cal G}^F}(R)/RC_{{\cal G}^F}(R)$ is a $\sigma$-fixed involution. Thus $N_G(P)/PC_G(P)$ has even order. That causes a contradiction.

Finally, suppose that $S={\rm D}_4(q)$ and $p=3$. We choose a simply connected simple algebraic group $\cal G$ with a Frobenius map $F: {\cal G}\rightarrow \cal G$ so that ${\cal G}^F/Z({\cal G}^F)\cong S$. In order to avoid repeat, we directly use the notation in the fourth paragraph in the proof to the algebraic group $\cal G$ here. We take the coxeter graph associated with $\Phi$ as in the 11th paragraph. Set $h=h_{r_1}(-1)h_{r_2}(-1)h_{r_3}(-1)\in T^F$ and $g=x_s(u)\in R$ for some nonzero $u\in F_q$.
Since $g^h=x_s(-u)$ and field and graph automorphisms on ${\cal G}^F$ fix $h$,
the image $\bar h$ of $h$ in $N_{{\cal G}^F}(R)/RC_{{\cal G}^F}(R)$ is an involution fixed by field and graph automorphisms on ${\cal G}^F$. The natural map ${\cal G}^F\rightarrow S$ induces a group isomorphism $N_{{\cal G}^F}(R)/RC_{{\cal G}^F}(R)\cong N_S(Q)/QC_S(Q)$, and field and graph automorphisms on $S$ can be lifted to field and graph automorphisms on ${\cal G}^F$. So the image $\tilde h$ of $\bar h$ in $N_S(Q)/QC_S(Q)$ is fixed by field and graph automorphisms on $S$.
Since $p=3$, by the structure of the automorphism group of $S$, we may assume that $G=S\rtimes \bar P$, where $\bar P$ is an automorphism subgroup of $S$ generated by suitable field and graph automorphisms.
Clearly $\tilde h$ lies inside $C_{N_S(Q)/QC_S(Q)}(\bar P)$. Then, as in the fourth paragraph, we use Lemma \ref{Extension} to prove that $N_G(P)/PC_G(P)$ has to be of even order. That causes a contradiction.  The proof is done.
\end{proof}

\noindent{\bf Proof of Theorem \ref{Almost}}. Suppose that $S$ is ${\rm B}_2(2)'$. In this case, $S$ is isomorphic to $A_6$ and Theorem \ref{Almost} follows from Lemma \ref{Alternating}.

Suppose that $S$ is ${\rm G}_2(2)'$. In this case, $S$ is isomorphic to $^2{\rm A}_2(3)$. Then Theorem \ref{Almost} follows from Lemma \ref{A} when $p\neq 3$ and from Lemma \ref{Defining} when $p=3$.

Suppose that $S$ is $^2{\rm G}_2(3)'$. In this case, $S$ is isomorphic to ${\rm PSL}_2(8)$ and then Theorem \ref{Almost} follows from Lemma \ref{A}.

Suppose that $S$ is $^2{\rm F}_4(2)'$. In this case, $G=S$ and $p$-Sylow automizers of $G$ have even order. So this case is excluded.

In the rest of the proof, we suppose that $S$ is not isomorphic to ${\rm B}_2(2)'$, ${\rm G}_2(2)'$, $^2{\rm G}_2(3)'$, or $^2{\rm F}_4(2)'$. Then Theorem \ref{Almost} follows from Corollary \ref{ADE} and Lemmas \ref{Alternating}-\ref{Defining}.


\end{document}